%% file: main.tex
\documentclass{amsart}
\usepackage{amsmath, mathtools}
\usepackage{amssymb}
\usepackage{blkarray}
\usepackage{bbm}
\usepackage{mathdots}
\usepackage{dsfont}
\usepackage{todonotes,lscape}
\usepackage{graphicx, hhline}
\usepackage{hyperref}
\usepackage{tikz}
\usetikzlibrary{calc}
\usepackage{tabularx}
\usepackage{caption,subcaption}
\usepackage{comment, appendix}

\usepackage[backend=biber,style=alphabetic]{biblatex}
\input{CatherinesDefns.tex}

\theoremstyle{definition}
\newtheorem{definition}{Definition}[section]
\newtheorem{example}[definition]{Example}
\newtheorem{remark}[definition]{Remark}

\theoremstyle{plain}
\newtheorem{theorem}[definition]{Theorem}
\newtheorem{lemma}[definition]{Lemma}

\newtheorem{corollary}[definition]{Corollary}
\newtheorem*{corollary*}{Corollary}

\begin{document}

\title{Eigenpolytope Universality and\\Graphical Designs}
\author{Catherine Babecki and David Shiroma}
\address{Department of Mathematics and Department of Computational and Mathematical Sciences, 
California Institute of Technology, 
Pasadena, CA 91125}
\email{cbabecki@caltech.edu, dshiroma@google.com}
\thanks{
Research partially supported by the Walker Family Endowed Professorship in Mathematics at the University of Washington.\\
{\em Corresponding Author}: Catherine Babecki}


\begin{abstract}
We show that the eigenpolytopes of graphs are universal in the sense that every polytope, up to affine equivalence, appears as the eigenpolytope of some positively weighted graph. 
We next extend the theory of graphical designs, which are quadrature rules for graphs, to positively weighted graphs. Through Gale duality for polytopes, we show a bijection between graphical designs and the faces of eigenpolytopes. This bijection proves the existence of graphical designs with positive quadrature weights, and upper bounds the size of a minimal graphical design.  Connecting this bijection with the universality of eigenpolytopes, we establish three complexity results: it is strongly NP-complete to determine if there is a graphical design smaller than the mentioned upper bound, it is NP-hard to find a smallest graphical design, and it is \#P-complete to count the number of minimal graphical designs.
\end{abstract}
\maketitle

\section{Introduction}
 This paper unites two combinatorial constructions arising from the \emph{combinatorial graph Laplacian}, namely \emph{eigenpolytopes} \cite{Godsil_GGP} and \emph{graphical designs} \cite{graphdesigns}, through the theory of \emph{Gale duality} for polytopes \cite{GaleOriginal, GrunbaumBook}.  
An eigenpolytope is a polytope arising from the eigenspaces of a graph Laplacian.
The structure and symmetries of these polytopes provide information about the structure and symmetries of the underlying graph, particularly when the underlying graph is highly structured (see, for instance \cite{Godsil_DistReg}). 
A graphical design is a quadrature rule for a graph. 
Classical numerical quadrature allows one to compute or approximate definite integrals of certain classes of functions on $\RR^n$ by sampling at finitely many points.  
Analogously, a graphical design is a rule to compute the average of certain Laplacian eigenvectors by sampling on a subset of graph vertices. 
Previous work on graphical designs often focuses on unweighted regular graphs (e.g. \cite{Golubev, cubescodes,designsGaleDuality, RekhaStefanRandomWalks}), where the eigenvectors of a graph are independent of the choice of graph Laplacian.
For regular, unweighted graphs, the bijection between graphical designs and eigenpolytopes we establish is precisely the one shown in \cite{designsGaleDuality}.
For more literature review and details on the connection between graphical designs and eigenpolytopes through Gale duality, we refer to \cite{designsGaleDuality}.  

We summarize the main results of this paper.
\begin{itemize}
 \item The Laplacian eigenspaces of positively weighted graphs are general: for any collection of orthogonal subspaces which spans $\RR^n$ and contains the span of the all-ones vector as a subspace, there exists a connected, positively weighted graph on $n$ vertices whose Laplacian eigenspaces are precisely the given subspaces 
 (Lemma \ref{lem: ONB to graph}). The proof reveals a polyhedral structure underlying the possible eigenvalues of a graph with a given eigenbasis. We use this result to prove that every polytope up to affine equivalence appears as an eigenpolytope of a positively weighted graph (Theorem \ref{lem: embed polytope algo}). 
   Moreover, these are constructive proofs with strongly polynomial time algorithms.
    \item We establish a combinatorial bijection between graphical designs and faces of eigenpolytopes using Gale duality (Theorem \ref{thm: main gale for D-A}). Polyhedral geometry then provides an upper bound on the size of the support of minimal graphical designs (Theorem \ref{thm:general upper bound}).  We note that this bijection holds when the combinatorial Laplacian is replaced by any symmetric operator with the all-ones vector as an eigenvector (Remark \ref{rem: any operator with ones will do}). 
    \item Through Gale duality and our algorithms, we establish complexity results for graphical designs as translations of polytope hardness results (Theorems \ref{thm: is facet bound tight complexity}, \ref{thm: smallest design complexity}, \ref{thm: count designs}). We provide a linear program which finds a graphical design in polynomial time with a guaranteed sparseness, though minimality is not guaranteed (Remark \ref{rem: polytime LP}).
\end{itemize}

We organize the paper as follows.
Section \ref{sec: eigenpolytopes} introduces eigenpolytopes. 
Section \ref{sec: algorithms} contains our algorithmic results showing the universality of Laplacian eigenspaces and eigenpolytopes.
Section \ref{sec: design definitions} establishes definitions and notation for graphical designs on positively weighted graphs.
Section \ref{sec: gale duality} briefly recounts the theory of Gale duality for polytopes and uses this machinery to connect graphical designs to the eigenpolytopes of the graph, a key structural result.
This correspondence provides a proof of existence and an upper bound on the size of a minimal graphical design.
Using our strongly polynomial time algorithms to create graphs with given eigenpolytopes, we establish the following complexity results in Section \ref{sec: complexity} by translating established hardness results for polytopes. It is strongly NP-complete to determine if there is a design smaller than the upper bound established in Section \ref{sec: gale duality}, it is NP-hard to find a smallest graphical design, and it is \#P-complete to count the number of support-minimal graphical designs.

{\bf Acknowledgments.} We thank Timothy Duff, Shayan Oveis-Gharan, Stefan Steinerberger, and Rekha Thomas for their feedback and guidance, as well as the reviewers for many helpful comments and suggestions.

 \section{Eigenpolytopes of Weighted Graphs} \label{sec: eigenpolytopes}
Let $G= ([n],E, w)$ be a connected graph with vertex set $[n]:=\{1,2,\ldots,n\}$, edge set $E$, and positive edge weights $w: E \to \RR_{> 0}$. 
 Let $A \in \RR^{n\times n}$ be the weighted \emph{adjacency matrix} of $G$, with $A_{ij} = w(ij)$ if $ij \in E$ and 0 otherwise, and $D$ be the weighted (diagonal) \emph{degree matrix} of $G$, with $D_{ii} = \deg i = \sum_{ij \in E} w(ij)$.
The \emph{(combinatorial) Laplacian} of $G$ is the symmetric matrix $L = D-A$. 
 Since  $L$ is positive semidefinite, it has non-negative eigenvalues and an orthogonal set of eigenvectors $\varphi_1, \ldots, \varphi_n$ that form a basis of $\RR^n$. For general references on spectral and algebraic graph theory, we refer the reader to \cite{ChungSpectral,GodsilRoyleBook, SpielmanBook}.
   
 Letting $\ones$ denote the all-ones vector in $\RR^n$, we always have that $L \ones = 0 $, since the $i$-th row of $A$ sums to $\deg(i) = \sum_{ij \in E} w(ij)$. Therefore, $\ones$ is an eigenvector of $L$ with eigenvalue $0$. Moreover, since $G$ is connected, the eigenspace of 0 is spanned by $\ones$; we denote this eigenspace by $\L_1.$
It is convenient to not normalize the eigenvectors $\varphi_i$ so that we may take $\phi_1 = \ones$ as the basis vector for $\L_1$.
We will refer to the eigenspaces and eigenvalues of $L$ as the eigenspaces and eigenvalues of the graph $G$. 
 We use the following running example to illustrate many definitions and results. The circled numbers in a graph are a labeling of the graph vertices. 
  
\begin{example} \label{ex: runing ex 1}
Consider the weighted graph in Figure~\ref{fig: running ex}.  The weights are given by $w(ij) = 1/|N(i)| + 1/|N(j)|$, where $|N(i)|$ is the number of edges incident to $i$.

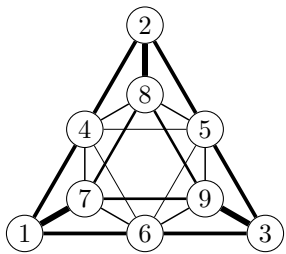
\begin{figure}[h]
\begin{center}
\begin{tikzpicture}[baseline={($ (current bounding box.west) - (0,1ex) $)},scale=.8]
        \tikzstyle{bk}=[circle, fill = white,inner sep= 2 pt,draw]
\node (v1) at (-2, 0) [bk] {1};
\node (v2) at (0, 3.464)  [bk] {2};
\node (v3) at (2, 0) [bk] {3};
\node (v4) at (-1, 1.732) [bk] {4};
\node (v5) at (1, 1.732) [bk] {5};
\node (v6) at (0, 0) [bk] {6};
\node (v7) at (-1, 0.577) [bk] {7};
\node (v8) at (0, 2.309) [bk] {8};
\node (v9) at (1, 0.577) [bk] {9};
\draw[line width=0.5mm] (v1) -- (v4) -- (v2) -- (v5) -- (v3) -- (v6) -- (v1);
\draw[line width=0.2mm]  (v4) -- (v8) -- (v5) -- (v9) -- (v6) -- (v7) -- (v4);
\draw[line width=0.03mm]  (v4) -- (v5) -- (v6) -- (v4);
\draw[line width=0.35mm]  (v7) -- (v8) -- (v9) -- (v7);
\draw[line width=0.8mm]  (v1) -- (v7);
\draw[line width=0.8mm]  (v2) -- (v8);
\draw[line width=0.8mm]  (v3) -- (v9);
\end{tikzpicture}  
\end{center}
\caption{A weighted graph on nine vertices. Thicker edges indicate greater edge weights.}
\label{fig: running ex}
\end{figure}
Its Laplacian $L$ is 
\[ 1/30
    \begin{bmatrix}
    46 & 0 & 0 & -15 & 0 & -15 & -16 & 0 & 0 \\
    0 & 46 & 0 & -15 & -15 & 0 & 0 & -16 & 0 \\
    0 & 0 & 46 & 0 & -15 & -15 & 0 & 0 & -16 \\
    -15 & -15 & 0 & 72 & -10 & -10 & -10 & -10 & 0 \\
    0 & -15 & -15 & -10 & 72 & -10 & 0 & -11 & -11 \\
    -15& 0 & -15 & -10 & -10 & 72 & -11 & 0 & -11 \\
    -16 & 0 & 0 & -11& 0 & -11 & 62 & -12 & -12 \\
    0 & -16& 0 & -11 & -11 & 0 & -12 & 62 & -12\\
    0 & 0 & -16 & 0 & -11 & -11 & -12 & -12 & 62
    \end{bmatrix}.
    \]
    
Table ~\ref{tab: running ex spectrum} shows the six distinct eigenvalues of $L$ in the left column, each followed on the right by a basis of orthogonal eigenvectors for that eigenspace. 
We note that $\lambda_2, \lambda_4, \lambda_6$ are the roots of $450x^3 - 3030x^2 + 6321x - 3844$ and $\l_3, \l_5$ are the roots of $75 x^2 - 340 x + 373$. We show decimal approximations for conciseness.
\end{example}
\begin{table}[h]
    \begin{center}\[ \scriptsize{
\begin{array}{c|ccccccccc |}
\text{vertex} & 1 & 2 & 3 & 4 & 5 & 6 & 7 & 8 & 9 \\
\hhline{= | =========|}
        \lambda_1 = 0 & 1 & 1 & 1 & 1 & 1 & 1 & 1 & 1 & 1 \\
        \hline
        \lambda_2 = 1.069 & -.6262 & .5616 & .0646 & -.0264 & .256 & -.2291 & -.3059 & .2744 & .0316 \\
                          & -.2870 & -.3988 & .6858 & -.2798 & .1171 & .1627 & -.1402 & -.1948 & .3350 \\
        \hline 
        \lambda_3 = 1.861 & .3193 & .3193 & .3193 & .1406 & .1406 & .1406 & -.4600 & -.4600 & -.4600 \\
        \hline
        \lambda_4 = 2.661 & .3248 & -.4014 & .0766 & .0507 & .2150 & -.2658 & -.4852 & .5996 & -.1144 \\
                          & .2760 & .1433 & -.4193 & -.2776 & .1827 & .0949 & -.4123 & -.2141 & .6263 \\
        \hline
        \lambda_5 = 2.672 & .3468 & .3458 & .3468 & -.4499 & -.4499 & -.4499 & .1032 & .1032 & .1032 \\
        \hline
        \lambda_6 = 3.003 & .0724 & -.0995 & .0271 & .1880 & .5015 & -.6894 & .2707 & -.3721 & .1015 \\
                          & .0731 & .0261 & -.0992 & -.6876 & .5066 & .1810 &.2734 & .0977 & -.3711
    \end{array} } \]
\end{center}
    \caption{The spectral information of the graph shown in Figure~\ref{fig: running ex}.}
    \label{tab: running ex spectrum}
\end{table}

For a given graph $G$ with $m$ eigenspaces, we denote a fixed arbitrary ordering of the eigenspaces as $ \L_1 = \spanset\{\ones\} < \ldots < \L_m$. Let $U$ denote a matrix whose rows are an eigenbasis of $L$, and let $I \subset [m]$ index a proper subset of eigenspaces of $G$. We use $U_{I}$ to denote the submatrix of $U$ consisting of all rows corresponding to the eigenspaces $\L_i$ for $i \in I$. Let $\cal U_I$ denote the collection of columns of $U_I$, which we note may occur with repetition. 

\begin{definition}[\cite{Godsil_GGP,designsGaleDuality}] \label{def:our eigenpolytopes}
Let $G= ([n],E,w)$ have $m$ eigenspaces $ \L_1 < \ldots < \L_m$, and let $I  \subset [m]$ index a set of eigenvalues of $G$. The polytope $P_{I} = \conv ( \cal U_{I} ) $ is a (Laplacian) {\em eigenpolytope} of $G$ for the eigenvalues indexed by $I$.
\end{definition}

Eigenpolytopes were first defined by Godsil \cite{Godsil_GGP} as above for single eigenvalues of the adjacency matrix $A$.  To the best of our knowledge, the straightforward generalization to multiple eigenvalues first appeared in \cite{designsGaleDuality}.
The eigenpolytope literature is often interested in the connection between the structure and symmetry of eigenpolytopes and the automorphism group of the graph (see for example, \cite{Godsil_DistReg,ChanGodsil}). The choice here and in \cite{designsGaleDuality} to
consider weighted graphs with no requirement of symmetry is atypical of the literature. Though the definition depends on a choice of eigenbasis, eigenpolytopes of a graph are well-defined up to affine equivalence, which is sufficient for our purposes.

\begin{example}
    Figure~\ref{fig: ex eigenpolytope} shows the eigenpolytope $P_{\{5,6\}}$ of the graph $G$ from Example~\ref{ex: runing ex 1}, using the same eigenspace ordering as in Table~\ref{tab: running ex spectrum}. Its vertices are given by the columns of the following submatrix extracted from Table~\ref{tab: running ex spectrum}, where the horizontal line separates $\L_5$ from $\L_6$. 
\[ \left[ \small{
\begin{array}{ccccccccc }
     .3468 & .3458 & .3468 & -.4499 & -.4499 & -.4499 & .1032 & .1032 & .1032 \\
        \hline
       .0724 & -.0995 & .0271 & .1880 & .5015 & -.6894 & .2707 & -.3721 & .1015 \\
          .0731 & .0261 & -.0992 & -.6876 & .5066 & .1810 &.2734 & .0977 & -.3711
    \end{array}  }\right] \]
\end{example}
\begin{figure}[h]
    \centering
    \begin{tikzpicture}[scale = .5]
\tikzstyle{bk}=[circle,draw =black, fill=white ,inner sep=1.5pt]
      \tikzstyle{middle}=[circle,draw =black, fill =black ,inner sep=1.5pt]
      \tikzstyle{frontcy}=[circle,draw =black, fill=black ,inner sep=2pt]
    \tikzstyle{front}=[circle,draw =black, fill=black ,inner sep=2pt]
        \tikzstyle{backcy}=[circle,draw =black, fill=black ,inner sep=1 pt]
  \tikzstyle{back}=[fill,circle, draw=black,fill=black, inner sep=1 pt]
    \tikzstyle{midcy}=[circle,draw =black, fill =black,inner sep=1.5pt]
    \fill[fill = gray!10] (0, 3.5) -- (5, 2) --  (2, -.3) -- (-3.75, 0) -- (0, 3.5);
        \fill[fill = gray!20] (7, 1)  -- (5, 2) --  (2, -.3) -- (6, 0) -- (7, 1) ;
            \fill[fill = gray!40] (6, 0) -- (7, -1)  --  (7, 1) -- (6, 0) ;    
           \fill[fill = black!30] (7, -1)  -- (5, -2) --  (2, -.3) -- (6, 0) -- (7, -1) ;
           \fill[fill = black!40]  (2, -.3) -- (-3.75, 0) -- (0,-2) --   (5, -2) -- (2, -.3) ;
\node (v1) at (6, 0) [front]{} ;
\node (v2) at (7, -1)  [middle] {};
\node (v3) at (7, 1) [middle] {};
\node (v4) at (0, 3.5) [midcy] {};
\node (v5) at (-3.75, 0) [frontcy] {};
\node (v6) at (0, -2) [middle] {};
\node (v7) at (2, -.3) [frontcy]{};
\node (v8) at (5, -2) [back] {};
\node (v9) at (5, 2) [backcy] {};
\node (1) at (6, .7) {1};
\node (2) at (7.5, -1.3)  {2};
\node (3) at (7.5, 1.3) {3};
\node (4) at (0.2, 4)  {4};
\node (5) at (-4.4, 0) {5};
\node (6) at (0.2, -2.6)  {6};
\node (7) at (1.7, .4)  {7};
\node (8) at (5.3, -2.4) {8};
\node (9) at (5.3, 2.4) {9};
\draw[thick] (v1) -- (v2) --(v3);
\draw (v7) --(v1) -- (v3) --(v9);
\draw[dashed] (v9) -- (v8);
\draw[thick] (v8) -- (v7);
\draw[thick] (v5) --(v6) -- (v8) -- (v2);
\draw[dashed] (v4) -- (v6);
\draw (v4) -- (v9) -- (v7) -- (v5) -- (v4);

\end{tikzpicture}
    \caption{The eigenpolytope $P_{\{5,6\}}$ of the running example }
    \label{fig: ex eigenpolytope}
\end{figure}
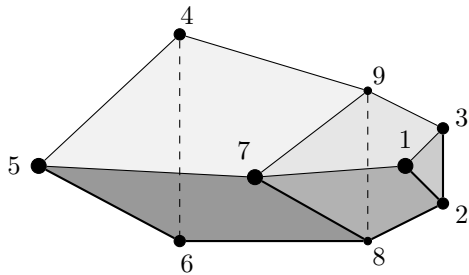

\section{Weighted Graphs From Polytopes} \label{sec: algorithms}

 In this section, we show that any polytope up to affine equivalence can appear as an eigenpolytope of a positively weighted graph.  
 We first show that any orthogonal basis for $\RR^n$ which includes $\ones $ as a basis vector, and any partition of this basis with $\{\ones\}$ as a part, give rise to a positively weighted graph $G = ([n], E, w)$ where each part of the partition spans a distinct Laplacian eigenspace of $G$. Our proof reveals a polyhedral structure underlying the valid choices of eigenvalues that define such graphs. These results are phrased in terms of rational numbers for the purpose of time-complexity statements, but the same algorithms are valid for real numbers.

 To state our complexity results, we recall some basics of encoding size. We reference a streamlined presentation tailored to our needs in \cite[Chapter 2]{SchrijverLPIPbook} and refer further to \cite{ComputersIntractability} for an in-depth treatment. 
 
The \emph{size} of a rational number $r = p/q$ in reduced form is \[\size(r) = 1 + \lceil \log_2 ( |p| +1 ) \rceil  + \lceil \log_2 ( |q| +1 ) \rceil ,\]
and the \emph{size} of a rational matrix $M \in \QQ^{n\times m}$ is 
\[\size(M) = mn + \sum_{i,j} \size (M_{ij}).\]
An algorithm is \emph{strongly polynomial time} if its run time is polynomial in size of the input, not just the dimension of the input. We will often present our input as a list of vectors $\cal V = \{ v_1,\ldots,  v_m\} \subset \QQ^n$ rather than as a matrix $V = [ v_1 \, \ldots \, v_m] \in \QQ^{n\times m}$. Clearly, $\size(\cal V) = \sum_{i=1}^m \size(v_i) = \size(V)$
 
We overload the notation $\Diag$; for a vector $v\in \RR^n$, $\Diag(v) \in \RR^{n\times n}$ is the diagonal matrix with $\Diag(v)_{ii} = v_i$, and for a matrix $M \in \RR^{n\times n}$,  $\Diag(M) \in \RR^{n\times n}$ is the diagonal matrix which forgets the off-diagonal entries of $M$.

 \begin{lemma} \label{lem: ONB to graph}
Let $\cal B = \{\phi_1 = \ones, \phi_2, \ldots, \phi_n\}$ be a rational orthogonal basis of $\RR^n$, partitioned as $\cal B= \pi_1 \sqcup \ldots \sqcup \pi_m$  with at least two parts such that $\pi_1 = \{\ones\}$. There is a connected graph $G$ with positive rational edge weights and $m$ Laplacian eigenspaces $\L_1, \ldots, \L_m$ such that $\pi_i$ spans $\L_i$. The graph $G$ can be constructed in strongly polynomial time, i.e.  polynomial in $\size(\cal B)$. 
\end{lemma}

\begin{proof}
    Let $B\in\QQ^{n\times n}$ have $i$-th row $\phi_i$.  The rows of $\tilde{B} = \Diag( \frac{1}{\|\phi_1\|}, \ldots, \frac{1}{\|\phi_n\|}) B$ form an orthonormal basis of $\RR^n$. We seek a rational matrix $M = \Diag(0, \l_2, \ldots, \l_n)$ with $\l_i > 0 $ and $\l_i = \l_j$ if and only if $\phi_i$ and $\phi_j$ lie in the same part of the partition so that $L = \tilde{B}^\top M \tilde{B}$ is the Laplacian of a positively weighted graph. 

Assuming such an $M$ exists, $A =  \Diag(\tilde{B}^\top M \tilde{B}) - \tilde{B}^\top  M \tilde{B}$ is then the adjacency matrix of the desired graph $G$, which has Laplacian eigenspaces $\L_1, \ldots, \L_m$ where $\L_i$ is spanned by $\pi_i$.  We first check several properties of this construction.
\begin{enumerate}
    \item $L $ is positive semidefinite: We see that $ L = (M^{1/2}\tilde{B})^\top  (M^{1/2} \tilde{B})$.
    \item $D =\Diag(\tilde{B}^\top M\tilde{B}) $ is the degree matrix of $G$:  Since $\tilde{B}^\top M \tilde{B}\ones = 0$, $D $ satisfies $D_{ii} = (\tilde{B}^\top M \tilde{B})_{ii} = -\sum_{j\neq i} (\tilde{B}^\top M \tilde{B})_{ij} =  \sum_{j\neq i} A_{ij} = \deg(i)$. 
    \item $G$ is connected: By construction, the multiplicity of 0 as an eigenvalue is 1. 
    \item $G$ has rational weights. For $i \neq j,$  
    \begin{align*}
 w_{ij} = -L_{ij} = -\sum_{k = 2}^n \frac{\phi_k(i)\phi_k(j)\lambda_k}{\phi_k^\top \phi_k} \in \QQ\end{align*}    
    \item It is polynomial in $\size(B)$ to compute $A$ from $B$ and $M$. The most expensive operation is multiplication of $n \times n$ matrices.
\end{enumerate}
    
It remains to show that we can find an $M$ with the given properties so that $A \geq 0$. For the rest of this proof, let $\ones \in \RR^{n - 1}$. We will prove that such an $M$ exists by showing that $M = \Diag(0, \ones^\top)$ lies in the interior of the constraints imposed by $A \geq 0$, and can be perturbed to satisfy all the needed conditions. We note that $M = \Diag(0, \ones^\top)$ corresponds to an equally weighted complete graph and solves the problem when $m =2$. 
Since $\lambda_1 = 0$, 
    $$L = \tilde{B}^\top M\tilde{B} = \sum_{k = 1}^n \frac{\lambda_k}{\|\phi_k\|^2} \phi_k\phi_k^\top = \sum_{k = 2}^n \frac{\lambda_k}{\|\phi_k\|^2} \phi_k\phi_k^\top.$$
        Since we require positive edge weights, we need that for $i \neq j$, $$A_{ij} = -L_{ij} = -\sum_{k = 2}^n \frac{\phi_k(i)\phi_k(j)\lambda_k}{\|\phi_k\|^2} \geq 0.$$
    
    This provides a system of $n(n - 1)/2$ linear inequalities constraining the choice of eigenvalues.
    Represent this system as $C\lambda \geq 0$, where $\lambda = (\l_2,\ldots,\l_n)\in \RR^{n - 1}$ and $C \in \RR^{n(n - 1)/2 \times (n - 1)}$.
    We note that $K= \{\lambda \in \RR^{n-1}_{+}: C \l \geq 0\}$ is a polyhedral cone.
    For some $ i\neq j$, let \[c = \left(\frac{-\phi_2(i)\phi_2(j)}{\|\phi_2\|^2 } ,\ldots, \frac{-\phi_n(i)\phi_n(j)}{\|\phi_n\|^2}\right)^\top\]  be a row of $C$. Since $\tilde{B} $ is orthonormal, the columns of $\tilde{B} $ are orthogonal. Therefore,
    \begin{align*}
        0 &= \left(\frac{1}{\sqrt{n} },  \frac{\phi_2(i)}{\|\phi_2\|}, \cdots ,\frac{\phi_n(i)}{\|\phi_n\|}\right) \left(\frac{1}{\sqrt{n} }, \frac{\phi_2(j)}{\|\phi_2\|}, \cdots ,\frac{\phi_n(j)}{\|\phi_n\|}\right)^\top\\
        & = \frac{1}{n} + \sum_{k = 2}^n \frac{\phi_k(i)\phi_k(j)}{\|\phi_k\|^2}
        = \frac{1}{n} - c^\top\ones.
    \end{align*} 
    
    Hence $c^\top\ones = \frac{1}{n} > 0$.
    Since $c$ was an arbitrary row of $C$, we see that $C\ones > 0$, implying that the ray spanned by $\ones$ lies in the interior of $K$.    
    \begin{figure}[h!] 
    \begin{tikzpicture}[scale = .4]
        \draw[fill=blue!10] (4, -2) -- (-6, -2) -- (-4, 1) -- (6, 1) -- cycle
            node[midway,right] {$c^\top \lambda \geq 0$};
        \draw[->] (0, 0) -- (0, 5)
            node[left] {$c$};
        \draw[->] (0, 0) -- (4, 4)
            node[above right] {$\ones$};
        \draw[style=dashed] (0, 0) -- (4, 0);
        \draw[style=dashed] (4, 4) -- (4, 0)
            node[midway,right] {$d_c$};
        \draw[very thin] (0, 1.1) arc (90:45:1.1) 
            node[midway,above] {\small $\theta_c$};
        \draw[very thin] (4, 2.9) arc (270:225:1.1) 
            node[midway,below] {\small $\theta_c$};
        \draw[very thin] (-0.25, 0) -- (-0.25, 0.25) -- (0, 0.25);
        \draw[very thin] (3.75, 0) -- (3.75, 0.25) -- (4, 0.25);
        \filldraw[black] (0,0) circle (2pt);
    \end{tikzpicture}
    \caption{Computing the distance from $\ones$ to $\partial K$}
     \label{fig: computing dist to boundary of cone}
    \end{figure}
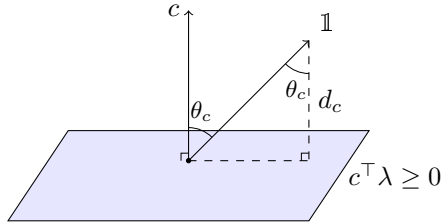
    Thus there is an $\varepsilon$-neighborhood of $\ones$ contained in $K$. Let $\theta_c$ be the angle between $\ones$ and $c$. 
    Observe from Figure~ \ref{fig: computing dist to boundary of cone} that the distance between $\ones$ and the hyperplane $c^\top \lambda \geq 0$ is
    $$d_c = \|\ones\|\cos\theta_c = \|\ones\| \cdot \frac{c^\top \ones}{\|c\|\|\ones\|} = \frac{1}{n\|c\|} > 0.$$
    Then such $\varepsilon \in \QQ$ can be computed in strongly polynomial time by rounding down each $d_c$ to two significant digits, taking the minimum over all possible $d_c$'s, and multiplying by some constant $\alpha < 1$ (e.g. $\alpha = 0.99$).
   For $k = 2, \ldots, n$, if $\phi_j\in \pi_k$, define $\lambda_j \coloneqq 1 + \delta_j$, where $\delta_j = \frac{\widetilde{k}\varepsilon}{n}$ and $\widetilde{k}$ is a rational approximation of $\sqrt{k}$ up to the precision needed to distinguish $\sqrt{m-1}$ from $\sqrt{m}$. Finding such an approximation takes time $\cal O(\log(m) + \text{precision})$ using binary search, and the precision depends polynomially on $m < n$.  Note that  $\sum_{k = 2}^m |\pi_k | = |\cal B| - 1 = n-1$. Then,
    \begin{align*}
        \left\|
            \lambda - \ones \right\|
       \leq \frac{\varepsilon}{n}\left(\sum_{k = 2}^m k|\pi_k|\right)^\frac{1}{2} \leq \frac{\varepsilon}{n}\left(n\sum_{k = 2}^m |\pi_k|\right)^\frac{1}{2} = \frac{\varepsilon}{n}\left(n(n - 1)\right)^\frac{1}{2} < \varepsilon.
    \end{align*}
    By construction,  $\l > 0$ since it is in the $\varepsilon$-neighborhood of $\ones$, and $\l \in \QQ^{n-1}$ satisfies the requirements imposed by the given eigenspace partition. Thus we have found the desired matrix $M = \Diag(0, \l_2, \ldots, \l_n) \in \QQ^{n\times n}$.
   \end{proof}

For a given orthogonal basis $\cal B$, the cone 
\[K = \{ \l = (\l_2, \ldots , \l_n) \geq 0: C\l \geq 0\}\] 
in the proof of Lemma~\ref{lem: ONB to graph} indexes all possible positively weighted graphs with $\cal B$ as its eigenvectors. 
The graph defined by $\l$ is missing edges if and only if $\l$ is in the boundary of $K$, and the facial structure of $K$ indexes the possible sparsity patterns of positively weighted graphs with these eigenvectors.

Furthermore, the \emph{braid arrangement} (see \cite[Example 3.11.11]{StanleyEC1}) partitions $K$ into chambers corresponding to the order of eigenvalues, which we depict for $n = 4$ in Figure~\ref{fig:braid arrangement}. Using nonstandard notation to be compatible with the notation of $K$, recall that the braid arrangement in $\RR^n$ is the collection of hyperplanes  
\[ \cal A_n = \{ H_{ij} = \{ \l = (\l_2, \ldots, \l_{n+1}) : \l_i = \l_j \} \}_{2 = i < j \leq n+1} .\]
Points lying on $H_{ij}$ correspond to graphs where the eigenspaces collapse into each other. As more hyperplanes of $\cal A_n$ intersect, an eigenvalue has higher multiplicity. 

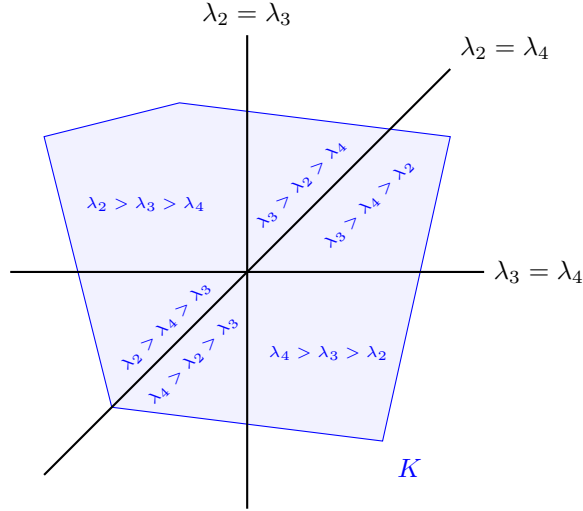
\begin{figure}[h] 
    \centering
    \begin{tikzpicture}[scale = .9]
         \draw[blue, fill=blue!5] (-2, -2) -- (-3, 2) -- (-1, 2.5) -- (3, 2) -- ( 2,-2.5) -- cycle;
        \draw[thick] (0, -3.5) -- (0, 3.5)
            node[above] {$\l_2 = \l_3$};
        \draw[thick] (-3, -3) -- (3, 3)
            node[above right] {$\l_2 = \l_4$};
        \draw[thick] (-3.5, 0) -- (3.5, 0)
            node[right] {$\l_3 = \l_4$};
        \node[blue] at (2.4, -2.9) {$K$};
        \node[blue] at (-1.5, 1) {\tiny $\l_2 > \l_3 > \l_4$};
        \node[blue, rotate = 45] at (.8, 1.3) {\tiny $\l_3 > \l_2 > \l_4$};
        \node[blue, rotate = 45] at (1.8, 1) {\tiny $\l_3 > \l_4 > \l_2$};
        \node[blue] at (1.2, -1.2) {\tiny $\l_4 > \l_3 > \l_2$};
        \node[blue, rotate = 45] at (-.8, -1.3) {\tiny $\l_4 > \l_2 > \l_3$};
        \node[blue,rotate = 45] at (-1.2, -.8) {\tiny $\l_2 > \l_4 > \l_3$};
    \end{tikzpicture}
    \caption{A cross-section of the braid arrangement $\cal A_3$ intersecting the cone $K$ of eigenvalues. The center point corresponds to $\l = \ones$, which defines the unweighted complete graph.}
    \label{fig:braid arrangement}
\end{figure}

We now use Lemma~\ref{lem: ONB to graph} to show that given a polytope $P$, we can construct a positively weighted graph in strongly polynomial time which has an eigenpolytope that is affinely equivalent to $P$.

\begin{theorem}
\label{lem: embed polytope algo}
Let $\cal V = \{v_1,\ldots, v_n\} \subset \QQ^d$ be such that $P  = \conv(\cal V)$ is full dimensional. We can create a connected graph $G= ([n], E, w)$ with $w \in \QQ^E_{>0}$ which has a Laplacian eigenpolytope affinely equivalent to $P$ in time polynomial in $\size(\cal V).$
    \end{theorem} 
    
\begin{proof}
  We may translate $P$ so that its centroid lies at the origin by mapping $v_i \mapsto v_i - \frac{1}{n} \sum_{i=1}^n v_i$. This is an affine transformation, and 
\[ \sum_{i=1}^n \left (v_i - \frac{1}{n} \sum_{i=1}^n v_i \right) = \sum_{i=1}^n v_i - n \frac{1}{n} \sum_{i=1}^n v_i = 0. \]
Thus, we may assume without loss of generality that $ \sum_{i=1}^n v_i =0$, or equivalently, that the rows of the matrix $ V = [ v_1 \ldots v_n ]$ are orthogonal to $\ones.$ 
Let $ X = [x_1\,\ldots\, x_n ]\in \QQ^{(n-d-1) \times n}$ have rows which form a basis for the kernel of the matrix $ [ \ones \,\,  V^\top]^\top$.  We can find $X$ in strongly polynomial time using Gaussian elimination (see \cite[Theorem 23.3]{SchrijverLPIPbook}).  
Let \[ \hat B =  \begin{bmatrix}
\ones^\top \\
V\\
X 
\end{bmatrix}. \]
We will think of the rows of $\hat B$ as the Laplacian eigenvectors of the graph we will construct. 
Perform Gram-Schmidt orthogonalization on the rows of $\hat B$ to obtain a rational orthogonal matrix 
\[B = \begin{bmatrix} \ones & \phi_2 & \ldots & \phi_n 
\end{bmatrix}^\top. \]
Gram-Schmidt orthogonalization is strongly polynomial time (see \cite[Section 1.4]{GrotschelLovaszSchrijverBook})
Note that $\spanset \{\phi_2,\ldots, \phi_{d+1}\} = \rowspan V$ because Gram-Schmidt sequentially preserves span. 
By Lemma~\ref{lem: ONB to graph}, we can construct a graph with rational, positive edge weights with the following three eigenspaces in strongly polynomial time: $\L_1 = \spanset\{\ones\},  \L_2 = \spanset\{\phi_2, \ldots, \phi_{d+1}\}, \L_3 = \spanset\{\phi_{d+2}, \ldots, \phi_{n}\}. $  
Because $V$ and $[ \phi_2 \ldots, \phi_{d+1}]^\top$ differ only by a change of basis, the eigenpolytope $P_{\{2\}}$ is affinely equivalent to $P = \conv (\cal V)$. 
\end{proof}

\begin{corollary}
Up to affine equivalence, every polytope with $n$ vertices appears as the eigenpolytope of a positively weighted graph on $n$ vertices.
\end{corollary}

\begin{example}
We illustrate Theorem~\ref{lem: embed polytope algo} by embedding the 16-cell, also known as the 4-dimensional cross polytope, as the eigenpolytope of a connected, positively weighted graph. A centered, orthogonal embedding of the 16-cell's vertices is
\[ V = \begin{bmatrix}
1 & 0 & 0 & 0 & -1 & 0 & 0 & 0 \\
0& 1 & 0 & 0 & 0 & -1 & 0 & 0 \\
0& 0& 1 & 0 & 0 & 0 & -1 & 0\\
0& 0& 0& 1 & 0 & 0 & 0 & -1 
\end{bmatrix}. \]
We next compute an orthogonal basis for the kernel of $[\ones, V^\top ]^\top$:
\[ X = \begin{bmatrix}
    -1&1& 0&  0&   -1& 1&  0&    0\\        
    -1& -1&  2  &    0&    -1&     -1&    2&   0\\
 - 1& -1& - 1& 3&  - 1&   - 1&  - 1&  3
\end{bmatrix}. \]
We set the eigenspace partition $\L_1 = \spanset\{ \ones\} ,\L_2 = \rowspan V, \L_3 = \rowspan X$. Following the procedure in Lemma~\ref{lem: ONB to graph}, we compute the following cone of eigenvalues for which the resulting graph will be positively weighted.
\[ K  = \left\{ (\l_2, \l_3) : \begin{bmatrix}
0 & 1/8 \\
1/2 & -3/8 
\end{bmatrix}  \begin{bmatrix}
\l_2\\
\l_3
\end{bmatrix} \geq 0 \right\}\]
A valid choice of eigenvalues is $\l_2 = 20$, $\l_3 = 24$. We depict the resulting complete weighted graph in Figure \ref{fig: our embedding of 16cell}.

\begin{figure}[h]
    \centering
    \begin{tikzpicture}[scale =.5]
\tikzstyle{bk}=[fill = white,circle, draw=black, inner sep=1.5 pt]
 \foreach \y in {1,2,3,4,5,6,7,8}
        {\node at (45*\y +45:3) (a\y) [bk] {\y};
        }
\draw[thick] (a1) -- (a2) -- (a3) -- (a4) -- (a5) -- (a6) -- (a7) -- (a8)  --(a1);
\draw[thick] (a1) -- (a3) -- (a5) -- (a7) -- (a1);
\draw[thick] (a2) -- (a4) -- (a6) -- (a8) -- (a2);
\draw[thick] (a1) -- (a4) -- (a7) -- (a2) -- (a5) -- (a8) -- (a3) -- (a6) -- (a1);
\draw[dashed] (a1) -- (a5) ;
\draw[dashed] (a2) -- (a6) ;
\draw[dashed] (a3) -- (a7) ;
\draw[dashed] (a4) -- (a8) ;
\end{tikzpicture}
    \caption{A graph which has the 16-cell as an eigenpolytope. Solid edges have weight 3, dashed edges have weight 1.}
    \label{fig: our embedding of 16cell}
\end{figure}
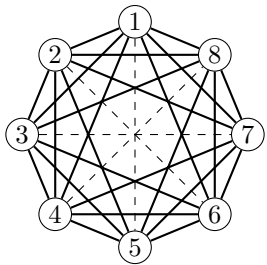

The algorithm in Theorem \ref{lem: embed polytope algo} embeds a polytope as an eigenpolytope using eigenvalues from the interior of the cone $K$, so the resulting graph will always be dense. To find a non-complete graph with the given eigenpolytope, one needs to find eigenvalues lying in the boundary of the cone $K$. We will use this idea to show that the edge graph of the 16-cell has the 16-cell as an eigenpolytope, a result due to \cite[Theorem 4.3]{Godsil_DistReg}.  The first inequality of $K$ ($\l_3/8 \geq 0$) disconnects the graph at equality. However, $ \l_2/2 -3\l_3/8 = 0$ has positive solutions, and    \[  \begin{bmatrix}
0 & 1/8 \\
1/2 & -3/8 
\end{bmatrix}  \begin{bmatrix}
6\\
8
\end{bmatrix} = \begin{bmatrix}
1\\
0
\end{bmatrix} ,\]
shows that $(6,8) \in K$. Thus the graph defined by the Laplacian \[L  = \begin{bmatrix}
    \ones  &  V^\top & X^\top
\end{bmatrix}  
\Diag( 0,6,6,6,6,8,8,8)
\begin{bmatrix}
    \ones^\top
    \\ V
    \\ X
\end{bmatrix}^{-1}
\]  is not dense and has the 16-cell as an eigenpolytope.  Working out the computations shows that this graph is truly the graph of the 16-cell, depicted in Figure~\ref{fig: 16 cell}.

\begin{figure}
    \centering
   \begin{tikzpicture}[scale =.5]
\tikzstyle{bk}=[fill= white,circle, draw=black, inner sep=1.5 pt]
 \foreach \y in {1,2,3,4,5,6,7,8}
        {\node at (45*\y +45:3) (a\y) [bk] {\y};}
\draw[thick] (a1) -- (a2) -- (a3) -- (a4) -- (a5) -- (a6) -- (a7) -- (a8)  --(a1);
\draw[thick] (a1) -- (a3) -- (a5) -- (a7) -- (a1);
\draw[thick] (a2) -- (a4) -- (a6) -- (a8) -- (a2);
\draw[thick] (a1) -- (a4) -- (a7) -- (a2) -- (a5) -- (a8) -- (a3) -- (a6) -- (a1);
\end{tikzpicture}
    \caption{The graph of the 16-cell,  for which the 16-cell is an eigenpolytope.}
    \label{fig: 16 cell}
\end{figure}
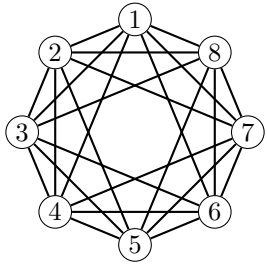

\end{example}

We show in Example~\ref{ex: cant expect unweighted graphs} that not every polytope may be embedded as the eigenpolytope of an unweighted graph on $n$ vertices.  The flexibility to use positive edge weights is crucial for this universality result.   

\begin{example} \label{ex: cant expect unweighted graphs}
    Up to affine equivalence, the quadrilateral
    \[P = \conv\left\{ \begin{bmatrix}
        0  \\
        0
    \end{bmatrix}, \begin{bmatrix}
        0  \\
        1
    \end{bmatrix},
    \begin{bmatrix}
        1  \\
        0
    \end{bmatrix},
    \begin{bmatrix}
        1  \\
        2
    \end{bmatrix}\right\}\]
    is not an eigenpolytope of any unweighted graph with four vertices. Following the procedure of Theorem~\ref{lem: embed polytope algo}, we find the following orthogonal basis of $\RR^4$, where the middle two rows correspond to the polytope: 
    \[B =  \begin{bmatrix}
    1& 1 &1& 1 \\ \hline
    -1& -1 &1& 1 \\
    -1& 1 &-2& 2 \\ \hline
    2& -2 &-1& 1 
    \end{bmatrix}.\]

Following Lemma~\ref{lem: ONB to graph}, any graph with this polytope as an eigenpolytope has 
\[L = B^\top \Diag(1/2, 1/2 ,1/10, 1/10)\Diag( 0,  a , b,  c) B,\]
where $a,b,c > 0$ and $a,b \neq c$, but we allow the possibility of $a = b$. 
Unraveling the multiplication, four unique expressions describe the edge weights of this graph. 
\begin{align*}
    a/4 -b/10 -2c/5 & && \text{edge }(1,2) \\
      -a/4 +b/5 -c/5 & && \text{edges }(1,3) ,(2,4)\\ 
      -a/4 -b/5 +c/5& && \text{edges }(1,4),(2,3) \\
       a/4 -2b/5 -c/10 &  && \text{edge }(3,4)
\end{align*}
    We check using MATLAB\cite{MATLAB} that for each $v \in \{0, -1\}^4$, the linear system 
    \[ \begin{bmatrix}
        1/4& -1/10 &-2/5 \\
         -1/4 &1/5& -1/5 \\
          -1/4 &-1/5 &1/5\\
            1/4& -2/5 &-1/10 
    \end{bmatrix} \begin{bmatrix}
        a \\
        b\\
        c
    \end{bmatrix} = v\]
    is either inconsistent or forces $ b=c$. Thus no unweighted graph has an eigenpolytope which is affinely equivalent to this polytope.
\end{example}

\begin{remark}
   To end this section, we contrast Theorem~\ref{lem: embed polytope algo} with  related results on \emph{Colin de Verdi{\`e}re matrices} (\cite{CdeVOriginal}, see also \cite{CdeVSurvey}). 
A Colin de Verdi{\`e}re matrix for an unweighted graph $G = ([n], E)$ is any matrix of the form $ M = \delta - \hat A$, where 
\begin{enumerate}
    \item $\hat A$ is the weighted adjacency matrix of a positively weighted graph $\hat G = ([n], E, w)$ on the same edge set as $G$, 
    \item $\delta$ is an arbitrary diagonal matrix, and
    \item $M$ has exactly one negative eigenvalue.
\end{enumerate}
By definition, $M$ is not positive semidefinite and need not have row sums equal to one.
For our purposes, we can think of these matrices as something like a Laplacian on a graph with not only positive edge weights, but also (arbitrary) vertex weights.  
Lov{\'a}sz and Schrijver show in \cite{CdeVLovaszSchrijver} that any 3-dimensional polytope is an eigenpolytope of its edge graph for every Colin de Verdi{\`e}re matrix with corank three. 
Izmestiev extends this construction in \cite{Ismestiev} to show that every polytope is an eigenpolytope of its edge graph for a specifically constructed Colin de Verdi{\`e}re matrix. 
These universality results for eigenpolytopes of Colin de Verdi{\`e}re matrices guarantee a serious amount of sparsity in the graphs needed to represent a given polytope as an eigenpolytope.  
We note, however, that $M$ is not exactly a Laplacian, and the underlying graph has arbitrary vertex weights. To find the specific Colin de Verdi{\`e}re matrix in Izmestiev's construction requires computing the volume of a polytope, which is a \#P-hard problem \cite{DyerFrieze}. Our construction generally produces a dense graph, but we do not require vertex weights, the matrix we use is the true combinatorial graph Laplacian, and the algorithm is strongly polynomial time. 
 \end{remark}

\section{Graphical Designs} \label{sec: design definitions}

A \emph{graphical design}\cite{graphdesigns} on $G= ([n], E, w)$ is a proper subset of graph vertices $S\subset [n]$ on which the global averages of certain graph eigenvectors agree with the weighted averages over just the subset $S$. If $\phi$ is such an eigenvector, that means we seek $S\subset [n]$ and quadrature weights $a_s\in \RR, \, s\in S$ such that 
\[ \frac{1}{n} \sum_{i=1}^n \phi(i) =   \sum_{s\in S} a_s \phi(s).\]
The \emph{support} of a vector $a \in \RR^n$ is $\supp(a) := \{i \in [n] \,:\, a_i \neq 0 \}$. 
A graphical design $S$ with weights $a_s$ may be identified with the vector $a \in \RR^n$ where $S = \supp(a)$.  We gloss over some motivation and explanation for the formal definition in Definition~\ref{def:graphical designs}. We refer those seeking more detail to \cite{designsGaleDuality}, where a nearly identical story is developed for regular unweighted graphs using the Laplacian $AD^{-1}$.

We recall our previous notation and establish some more.  For a given graph $G$ with $m$ eigenspaces, we denote a fixed arbitrary ordering of the eigenspaces as $ \L_1 = \spanset\{\ones\} < \ldots < \L_m$. The rows of the matrix $U$ form an eigenbasis of $L$, and for $I \subset [m]$, we use $U_{I}$ to denote the submatrix of $U$ consisting of all rows corresponding to the eigenspaces $\L_i$ for $i \in I$. Lastly, $\cal U_I$ denotes the collection of columns of $U_I$.  We use $\boldsymbol k$ to denote the particular index set $\{2,\ldots, k\} \subset [m]$. For a subset $S \subset [n]$, define $\ones_S$ by $\ones_S(i) = 1$ if $i \in S$ and 0 otherwise.

\begin{definition}[$k$-graphical designs: see \cite{graphdesigns,cubescodes,designsGaleDuality}]
\label{def:graphical designs}
Suppose $G= ([n],E,w)$ has eigenspaces $  \L_1 =\spanset \{\ones\}< \ldots < \L_m.$
\begin{enumerate}
    \item A \emph{weighted $k$-graphical design} of $G$ is a subset $S \subset [n]$ and real weights $(a_s \neq 0: s\in S)$ such that $U_{\boldsymbol k}a = 0$.
    \item If additionally $ a\geq 0 $, we call $S$ a \emph{positively weighted $k$-graphical design}.
    \item If  $a=\ones_S$ then $S$ is a \emph{combinatorial $k$-graphical design}.
\end{enumerate}
\end{definition}

Graphical designs were first introduced by Steinerberger in \cite{graphdesigns}, which considered averaging eigenvectors, not entire eigenspaces, of the operator $AD^{-1} - I$. The issue of eigenspace multiplicity was left open; the first author resolved this problem in \cite{cubescodes} by amending the definition to consider entire eigenspaces of $AD^{-1} - I$, but investigated only combinatorial $k$-graphical designs. The definition that appears above is a slight generalization of \cite{designsGaleDuality}[Definition 2.2], which considered only unweighted regular graphs and the operator $AD^{-1}$. For regular graphs, $AD^{-1} $, $D-A$, and most other graph operators carry equivalent spectral information because the matrices differ by only an affine transformation.

We only consider graphical designs up to scaling. If $U_{\boldsymbol k}a = 0$, then any scaling of $a$ is also in the kernel of $U_{\boldsymbol k}$. For this reason we leave $\L_1$ out of the definition, as any vector $a$ can be scaled to average this eigenspace.
We also note that no proper subset is an $m$-graphical design -- this follows by the same argument as in \cite[Lemma 2.5]{designsGaleDuality}.   Though $\ones$ is always a $k$-graphical design for all $k\in [m]$, the purpose of a graphical design is, in a sense, to compress data on a graph, so we restrict our attention to proper subsets of graph vertices.

We will often abbreviate $k$-graphical designs as $k$-designs.  
We note that there are two kinds of weights at play in our setup: the edge weights $w$ of the graph, and the quadrature weights $a_s$ which define the design.  Here, we will only consider positively weighted or combinatorial designs. Quadrature rules with negative weights are typically undesirable, as they can lead to nonconvergence and instability \cite{HuybrechsInstability}.

\begin{example} \label{ex: running ex pt2}
Consider the graph from Example~\ref{ex: runing ex 1} shown in Figure~\ref{fig: running ex} with eigenspaces ordered as labeled.
The quadrature weights $$a_1 =a_3 = .0342,\, a_2 = .1111, \, a_6 = .5328, \, a_8 = .2876,$$ show that $S = \{1,2,3,6,8\}$ averages the basis vectors given for $\L_1, \L_2, \L_3, \L_4$, which is to say that $U_{\{2,3,4\}} a = 0$. Thus these vertices with these weights form a positively weighted $4$-graphical design, shown in Figure~\ref{fig: running ex 4design} along with the full matrix $U_{\{2,3,4\}}$. In this ordering of the eigenspaces, there are no combinatorial 4-designs. \end{example}

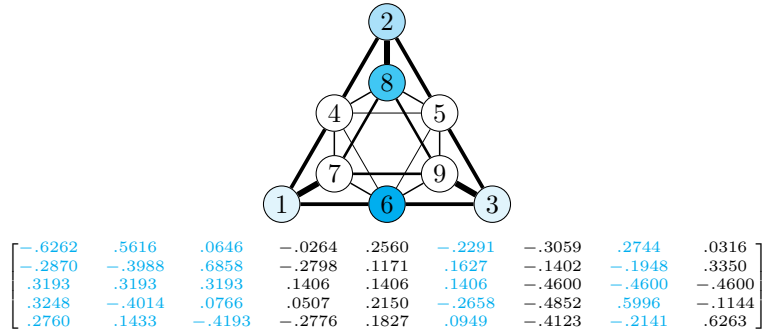
\begin{figure}[h!]
\begin{center}
\begin{tikzpicture}[baseline={($ (current bounding box.west) - (0,1ex) $)},scale = .7]
    \tikzstyle{bk}=[circle, fill = white,inner sep= 2 pt,draw]
    \tikzstyle{most}=[circle, fill = cyan,inner sep= 2 pt,draw]
    \tikzstyle{half}=[circle, fill = cyan!60,inner sep= 2 pt,draw]
    \tikzstyle{quarter}=[circle, fill = cyan!30,inner sep= 2 pt,draw]
    \tikzstyle{little}=[circle, fill = cyan!10,inner sep= 2 pt,draw]
\node (v1) at (-2, 0) [little] {1};
\node (v2) at (0, 3.464)  [quarter] {2};
\node (v3) at (2, 0) [little] {3};
\node (v4) at (-1, 1.732) [bk] {4};
\node (v5) at (1, 1.732) [bk] {5};
\node (v6) at (0, 0) [most] {6};
\node (v7) at (-1, 0.577) [bk] {7};
\node (v8) at (0, 2.309) [half] {8};
\node (v9) at (1, 0.577) [bk] {9};
\draw[line width=0.5mm] (v1) -- (v4) -- (v2) -- (v5) -- (v3) -- (v6) -- (v1);
\draw[line width=0.2mm]  (v4) -- (v8) -- (v5) -- (v9) -- (v6) -- (v7) -- (v4);
\draw[line width=0.03mm]  (v4) -- (v5) -- (v6) -- (v4);
\draw[line width=0.35mm]  (v7) -- (v8) -- (v9) -- (v7);
\draw[line width=0.8mm]  (v1) -- (v7);
\draw[line width=0.8mm]  (v2) -- (v8);
\draw[line width=0.8mm]  (v3) -- (v9);
\end{tikzpicture} \\
\tiny{\[\begin{bmatrix} 
\textcolor{cyan}{-.6262} & \textcolor{cyan}{.5616} & \textcolor{cyan}{.0646} & -.0264 & .2560 & \textcolor{cyan}{-.2291} & -.3059 & \textcolor{cyan}{.2744} & .0316 \\
\textcolor{cyan}{ -.2870} & \textcolor{cyan}{-.3988} & \textcolor{cyan}{.6858} & -.2798 & .1171 & \textcolor{cyan}{.1627} & -.1402 & \textcolor{cyan}{-.1948} & .3350 \\
\textcolor{cyan}{ .3193} & \textcolor{cyan}{.3193} & \textcolor{cyan}{.3193} & .1406 & .1406 & \textcolor{cyan}{.1406} & -.4600 & \textcolor{cyan}{-.4600} & -.4600 \\
 \textcolor{cyan}{.3248} & \textcolor{cyan}{-.4014} & \textcolor{cyan}{.0766} & .0507 & .2150 & \textcolor{cyan}{-.2658} & -.4852 & \textcolor{cyan}{.5996 }& -.1144 \\
\textcolor{cyan}{ .2760} & \textcolor{cyan}{.1433 }&\textcolor{cyan}{ -.4193} & -.2776 & .1827 & \textcolor{cyan}{.0949} & -.4123 & \textcolor{cyan}{-.2141} & .6263 \\
\end{bmatrix}\]}
\end{center}

\caption{A positively weighted 4-graphical design on a positively weighted graph. A more saturated color on a vertex indicates a larger vertex weight in the design. The matrix $U_{\{2,3,4\}}$ is also shown, with the relevant columns highlighted in cyan.}
\label{fig: running ex 4design}
\end{figure}

\section{A Bijection through Gale Duality} \label{sec: gale duality}

We now use Gale duality to connect graphical designs and eigenpolytopes, extending the results in \cite{designsGaleDuality} for unweighted regular graphs to positively weighted graphs without regularity assumptions. We begin with a brief introduction to Gale duality \cite{GaleOriginal}; see \cite[Chapter 6]{Ziegler} or \cite[Section 5.4]{GrunbaumBook} for more details.

A \emph{vector configuration} $\cal V  =(v_1, \ldots, v_n) $ is a collection of real vectors with possible repetitions. We denote the matrix $ [v_1 \, \ldots \, v_n ]$ by $V$.
A {\em dependence} on $\cal V$ is  a vector $a\in \RR^n$ such that $Va = 0$.  A support-minimal dependence is called a \emph{circuit} of $\mathcal{V}$.

\begin{theorem}[Gale duality: see \cite{Ziegler, GrunbaumBook}] \label{thm: Gale duality} Suppose  $$\cal V = (v_1, \ldots, v_n) ,\,\, v_i \in \RR^{n-d-1} \textup{ and } \cal V^*=  (v^*_1, \ldots, v^*_n), \,\,v_i^* \in  \RR^{d+1}$$ are vector configurations satisfying the following properties.
\begin{enumerate}
    \item $V $ and $V^*$ are full rank matrices.
    \item $V(V^*)^\top =0$
    \item $\ones$ is the first row of $V^*$.
\end{enumerate} 
Then, for any $I \subseteq  [n]$,  $ \conv \{v_i^\ast: i \in [n] \setminus I\}$ is a face of $P^* = \conv(\mathcal{V}^\ast)$ if and only if $0$ is in the relative interior of $\conv\{v_i: i \in I\}$.
\end{theorem}

For an index set $I \subseteq [n]$, $0$ is in the relative interior of $\conv\{v_i: i \in I\}$ if and only if there exists $c \geq 0$ with $\supp(c) = I$ and $Vc = 0$. Equivalently $c$ is a positive dependence of $\mathcal{V}$ with $\supp(c)=I$. 
Therefore we have the following consequence. 

\begin{corollary}[see {\cite[Theorem~5.4.1]{GrunbaumBook}} ]
\label{cor:faces and dependencies} Let $\cal V$ and $\cal V^*$ be as in Theorem \ref{thm: Gale duality}. 
For any $I \subseteq  [n]$,  $ \conv \{v_i^\ast: i \in [n] \setminus I\}$ is a face (facet) of $P^\ast =\conv(\mathcal{V}^\ast)$ if and only if $I$ is the support of a positive dependence (circuit) of $\mathcal{V}$. 
\end{corollary}

We now have the tools to connect $k$-graphical designs to eigenpolytopes. Recall that $G$ has $m$ eigenspaces, and $\boldsymbol{k} = \{2, \ldots, k\}\subset [m]$. 

\begin{theorem}[See also {\cite[Theorem 3.8]{designsGaleDuality}}] \label{thm: main gale for D-A}
Let $G = ([n], E, w)$ be a connected weighted graph with $m$ eigenspaces $\L_1 < \ldots < \L_m$. 
A set $S  \subset [n]$ is a (minimal) positively weighted $k$-graphical design of $G$ if and only if $[n] \setminus S$ indexes the elements of $\mathcal{U}_{[m] \setminus \boldsymbol{k}}$ which lie on a face (facet) of the eigenpolytope $P_{[m] \setminus \boldsymbol{k}}$.
\end{theorem}

\begin{proof}
By definition, positively weighted $k$-graphical designs of $G$ are positive dependences on the matrix $U_{\boldsymbol {k}}$. 
The matrices $U_{\boldsymbol {k}}$ and $U_{[m] \setminus \boldsymbol{k}}$ satisfy the hypotheses of Corollary \ref{cor:faces and dependencies}: they are full rank, the first row of $U_{[m] \setminus \boldsymbol{k}}$ is the eigenvector $\ones$, and by the orthogonality of eigenspaces, $U_{\boldsymbol {k}} U_{[m] \setminus \boldsymbol{k}}^\top =0.$ 
Therefore the faces (facets) of $P_{[m] \setminus \boldsymbol{k}}$ are dual to the dependences (circuits) of $\cal U_{\boldsymbol {k}}$ in the sense of Gale duality: $S$ indexes the elements of $\cal U_{[m] \setminus \boldsymbol{k}}$ lying on a face (facet) of $P_{[m] \setminus \boldsymbol{k}}$ if and only if $[n] \setminus S $ indexes the support of a (minimal) dependence on $\cal U_{\boldsymbol {k}}$.
\end{proof}
If one allows negative quadrature weights, arbitrary graphical designs are (essentially) in bijection with the complements of hyperplanes slicing through the interior of the corresponding eigenpolytope. This broader story is briefly touched on in \cite[Example 3.18]{designsGaleDuality} and its preceding commentary.

\begin{example}
We illustrate the bijection established in Theorem~\ref{thm: main gale for D-A}.
 Figure~\ref{fig: weighted ex P56} re-exhibits the 4-graphical design, which is a circuit of $\cal U_{\{2,3,4\}}$, along with the Gale dual eigenpolytope $P_{\{1,5,6\}} \simeq P_{\{5,6\}}$.  The cyan facet is dual to the design.
\end{example}

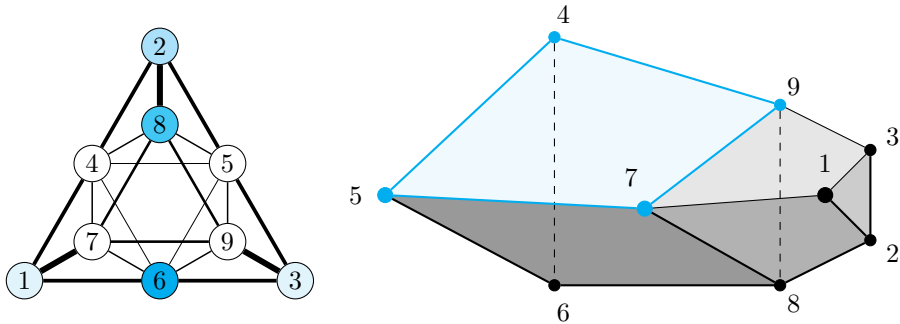
\begin{figure}[h!]
\begin{align*}
 &\begin{tikzpicture}[baseline={($ (current bounding box.west) - (0,1ex) $)},scale = .9]
    \tikzstyle{bk}=[circle, fill = white,inner sep= 2 pt,draw]
    \tikzstyle{most}=[circle, fill = cyan,inner sep= 2 pt,draw]
    \tikzstyle{half}=[circle, fill = cyan!60,inner sep= 2 pt,draw]
    \tikzstyle{quarter}=[circle, fill = cyan!30,inner sep= 2 pt,draw]
    \tikzstyle{little}=[circle, fill = cyan!10,inner sep= 2 pt,draw]
\node (v1) at (-2, 0) [little] {1};
\node (v2) at (0, 3.464)  [quarter] {2};
\node (v3) at (2, 0) [little] {3};
\node (v4) at (-1, 1.732) [bk] {4};
\node (v5) at (1, 1.732) [bk] {5};
\node (v6) at (0, 0) [most] {6};
\node (v7) at (-1, 0.577) [bk] {7};
\node (v8) at (0, 2.309) [half] {8};
\node (v9) at (1, 0.577) [bk] {9};
\draw[line width=0.5mm] (v1) -- (v4) -- (v2) -- (v5) -- (v3) -- (v6) -- (v1);
\draw[line width=0.2mm]  (v4) -- (v8) -- (v5) -- (v9) -- (v6) -- (v7) -- (v4);
\draw[line width=0.03mm]  (v4) -- (v5) -- (v6) -- (v4);
\draw[line width=0.35mm]  (v7) -- (v8) -- (v9) -- (v7);
\draw[line width=0.8mm]  (v1) -- (v7);
\draw[line width=0.8mm]  (v2) -- (v8);
\draw[line width=0.8mm]  (v3) -- (v9);
\end{tikzpicture} &
\begin{tikzpicture}[baseline={($ (current bounding box.west) - (0,1ex) $)},scale = .6]
          \tikzstyle{frontcy}=[circle,draw =cyan, fill=cyan ,inner sep=2pt]
    \tikzstyle{front}=[circle,draw =black, fill=black ,inner sep=2pt]
        \tikzstyle{backcy}=[circle,draw =cyan, fill=cyan ,inner sep=1.5 pt]
  \tikzstyle{back}=[fill,circle, draw=black,fill=black, inner sep=1.5 pt]
        \fill[fill = cyan!5] (0, 3.5) -- (5, 2) --  (2, -.3) -- (-3.75, 0) -- (0, 3.5);
        \fill[fill = gray!20] (7, 1)  -- (5, 2) --  (2, -.3) -- (6, 0) -- (7, 1) ;
            \fill[fill = gray!40] (6, 0) -- (7, -1)  --  (7, 1) -- (6, 0) ;    
           \fill[fill = black!30] (7, -1)  -- (5, -2) --  (2, -.3) -- (6, 0) -- (7, -1) ;
           \fill[fill = black!40]  (2, -.3) -- (-3.75, 0) -- (0,-2) --   (5, -2) -- (2, -.3) ;
\node (v1) at (6, 0) [front]{} ;
\node (v2) at (7, -1)  [back] {};
\node (v3) at (7, 1) [back] {};
\node (v4) at (0, 3.5) [backcy] {};
\node (v5) at (-3.75, 0) [frontcy] {};
\node (v6) at (0, -2) [back] {};
\node (v7) at (2, -.3) [frontcy]{};
\node (v8) at (5, -2) [back] {};
\node (v9) at (5, 2) [backcy] {};
\node (1) at (6, .7) {1};
\node (2) at (7.5, -1.3)  {2};
\node (3) at (7.5, 1.3) {3};
\node (4) at (0.2, 4)  {4};
\node (5) at (-4.4, 0) {5};
\node (6) at (0.2, -2.6)  {6};
\node (7) at (1.7, .4)  {7};
\node (8) at (5.3, -2.4) {8};
\node (9) at (5.3, 2.4) {9};
\draw[thick] (v1) -- (v2) --(v3);
\draw (v7) --(v1) -- (v3) --(v9);
\draw[dashed] (v9) -- (v8);
\draw[thick] (v8) -- (v7);
\draw[thick] (v5) --(v6) -- (v8) -- (v2);
\draw[dashed] (v4) -- (v6);
\draw[thick, cyan] (v4) -- (v9) -- (v7) -- (v5) -- (v4);
\end{tikzpicture}
\end{align*}
\caption{A positively weighted 4-graphical design and its corresponding facet on the Gale dual eigenpolytope. A more saturated color on a graph vertex corresponds to a larger quadrature weight.}
\label{fig: weighted ex P56}
\end{figure}

\begin{remark} \label{rem: any operator with ones will do}
The bijection established in this section remains true if $D-A$ is replaced by any symmetric operator which has $\L_1 = \spanset\{\ones\}$ as an eigenspace; for instance, the symmetric normalized adjacency matrix  $D^{-1/2} A D^{-1/2} $. See \cite{ChungSpectral} for more on various graph Laplacians.
\end{remark}

\begin{remark} 
One could also define graphical designs for positively weighted graphs using the eigenvectors of $AD^{-1}$, however, the framework provided by Gale duality and eigenpolytopes does not immediately apply.  If the graph is not regular, the eigenspaces of $AD^{-1}$ need not be orthogonal, hence the matrices $U_{\boldsymbol {k}}$ and $U_{[m] \setminus \boldsymbol {k}}$ need not satisfy the conditions of Theorem \ref{thm: Gale duality}. 
This can be worked around with some care if one considers the dual configurations given by $U_{\boldsymbol {k}}$ and $U_{[m] \setminus \boldsymbol {k}} D^{-1}$. There are many details to this, which we do not elaborate on here.
\end{remark}

Theorem~\ref{thm: main gale for D-A} provides an upper bound to the size of a minimal $k$-graphical design, similar to \cite[Theorem 3.14]{designsGaleDuality}.

\begin{theorem}\label{thm:general upper bound} 
Let $G = ([n],E,w)$ be a connected, positively weighted graph with eigenspaces
$\Lambda_1 < \cdots < \Lambda_m$. 
For every $k \in [m-1]$ there is a positively weighted $k$-graphical design with at most $\sum_{i=1}^k \dim \Lambda_i$ vertices. \end{theorem}

\begin{proof}
Because $\dim (P_{[m]\setminus{\boldsymbol k}}) = \sum_{i=k+1}^m \dim \Lambda_i, $ any facet of $P_{[m]\setminus{\boldsymbol k}}$ has at least $\sum_{i=k+1}^m \dim \Lambda_i$ distinct vertices. By Theorem~\ref{thm: main gale for D-A}, any minimal positively weighted $k$-design then has at most $n- \sum_{i=k+1}^n \dim \Lambda_i =  \sum_{i=1}^k \dim \Lambda_i$ vertices.
\end{proof} 

We note that this bound may be tight. The example of an unweighted regular graph and an eigenspace ordering for which this bound is tight for every $k$ given in \cite[Example 3.16]{designsGaleDuality} is also valid here. 

\begin{example} We return to our running example for the last time.
In the ordering specified in Example~\ref{ex: running ex pt2}, Theorem~\ref{thm:general upper bound} states that a minimal positively weighted 4-design has at most $6$ vertices and is dual to a triangular facet of $P_{\{1,5,6\}}$. This eigenpolytope also has  quadrilateral facets, which correspond to designs with $5 < 6$ vertices (see Figure~\ref{fig: weighted ex P56}).
\end{example}

We conclude this section with a few structural remarks about graphical designs that follow from the Gale duality bijection.

\begin{remark}
Fix an ordering $\L_1< \ldots<\L_m$ of a graph $G = ([n],E,w)$. The $i$-th element of $\cal U_{[m]\setminus{\boldsymbol k}}$ lies in the interior of the eigenpolytope $P_{[m]\setminus{\boldsymbol k}}$ if and only if $i$ is in the support of every $k$-graphical design.  
\end{remark}

We recall a quick result about combinatorial designs.

\begin{lemma}[Lemma 2.4 of \cite{designsGaleDuality}]
\label{lem: comb complements}
If $S \subset [n]$ is a combinatorial $k$-design, then so is $[n] \setminus S.$
\end{lemma}

\begin{proof}
Because $ U_{\boldsymbol{k}} \perp \ones$, 
\[   U_{\boldsymbol{k}} \ones_S = 0 \iff   U_{\boldsymbol{k}} (\ones - \ones_S) = 0 \iff  U_{\boldsymbol{k}} \ones_{[n]\setminus S} = 0 \]
\end{proof}

\begin{remark}
 If $S \subset [n]$ is a combinatorial $k$-design, then $S$ and $[n]\setminus S$ partition $P_{[m]\setminus{\boldsymbol k}}$ into two disjoint faces. Therefore, if $G$ has a combinatorial $k$-design, no vertices lie in the interior of $P_{[m]\setminus{\boldsymbol k}}$. 
\end{remark}

We speculate that the symmetry needed for a combinatorial design to exist prevents a graph vertex from being `so important' that it is in every design.

\section{Complexity Results} \label{sec: complexity}

The hardness of some algorithmic problems for graphical designs follows by uniting all of the established theory. We first recall some basics of polytopes; see \cite{GrunbaumBook} or \cite{Ziegler} for more. 
A fundamental result of polyhedral combinatorics is that there are two equivalent definitions of a polytope. 
A polytope $P$ can be written as the convex hull of finitely many points $P = \conv(\{v_1, \ldots, v_n\} )$ for some $v_i \in \QQ^m$, known as a $\cal V$-description. If $v_i \in \QQ^{m}, $ the size of a $\cal V$-polytope is $\size(P) = \sum_{i=1}^n \size(v_i)$.
A polytope can also be written as the bounded intersection of finitely many half-spaces, $P =  \{ x \in \RR^m: Ax \leq b\}$ for some $A \in \RR^{n \times m}$ and $b \in \RR^n$, known as an $\cal H$-description. 
If $A \in \QQ^{n\times m}$ and $b \in \QQ^n$, the size of an $\cal H$-polytope is $\size(P) = \size(A) + \size(b)$.

We will assume $0$ is in the interior of every polytope.
If $0$ is in the interior of $P = \conv( \{v_1,\ldots, v_n\})$, then $P^* := \bigcap_{i=1}^n \{ x: v_i^\top x \leq 1\}$ is also a polytope, known as the \emph{dual polytope} to $P$. We note that $(P^*)^* = P$. 
A $d$-dimensional polytope is \emph{simple} if every vertex is incident to exactly $d$ edges, or equivalently, exactly $d$ facets. We recall that a polytope is \emph{simplicial} if each of its facets is a simplex. Simplicial polytopes and simple polytopes are dual in this sense: if $P = \conv( \{v_1,\ldots v_n\})$ is simple (resp. simplicial), then $P^* = \bigcap_{i=1}^n \{ x: v_i^\top x \leq 1\}$ is simplicial (resp. simple).

 It was established independently in  \cite{ChandrasekaranVertexEnum} and \cite{DyerVertexEnum} that it is NP-complete to determine whether an $\cal H$-polytope is simple. The improvement to strong NP-completeness is due to \cite{FukudaVertexEnum}.

\begin{lemma}[\cite{ChandrasekaranVertexEnum,DyerVertexEnum, FukudaVertexEnum}] \label{lem: H-description simple} The following decision problem is strongly NP-complete.

\textbf{Input:} A rational $\cal H$-polytope $P = \{x: Ax \leq b$\}, $A \in \QQ^{m \times n}$, $b \in \QQ^m$.

\textbf{Question:} Is $P$ nondegenerate (simple)?
\end{lemma}

Passing to the dual polytope provides an immediate corollary.

\begin{corollary} \label{cor: V-description simplicial}
The following decision problem is strongly NP-complete.

\textbf{Input:} A rational $\cal V$-polytope $P = \conv(\{v_1,\ldots v_n\}) $, $v_i \in \QQ^m$

\textbf{Question:} Is $P$ simplicial?
\end{corollary}

We now use this corollary to show that it is strongly NP-complete to determine whether a graph has a positively weighted $k$-design of size smaller than the facet bound shown in Theorem~\ref{thm:general upper bound}.

\begin{theorem} \label{thm: is facet bound tight complexity}
The following decision problem is strongly NP-complete. 

\textbf{Input:} \begin{enumerate}
    \item The Laplacian matrix $L= D-A \in \QQ^{n\times n}$ of a connected, positively weighted graph $G= ([n],E,w)$ which has $m$ Laplacian eigenspaces $\L_1 = \spanset \{ \ones\} < \L_2 <\ldots < \L_m$, and
    \item an integer $k \in \{2,\ldots, m-1\}$. 
\end{enumerate} 

\textbf{Question:} Is there a positively weighted $k$-graphical design  of $G$ consisting of fewer than $\sum_{i=1}^k \dim (\L_i)$ vertices?
\end{theorem}
\begin{proof}
We first show that this problem is in NP. Given a weight vector $a \geq 0$ with $\supp(a) = S \subset [n]$, it is polynomial time in $n$ to check whether $a$ defines a $k$-graphical design by checking if $U_{\boldsymbol {k}} a = 0$ and further whether $|S| < \sum_{i=1}^k \dim (\L_i).$
 
To show strong NP-completeness, we will show that the decision problem in Corollary~\ref{cor: V-description simplicial} is a sub-problem of the decision problem in the theorem statement. Let $P = \conv(\{v_1, \ldots, v_n\}\subset \QQ^d)$ be a $\cal V$-polytope with $d \leq n-2$. We note that if $d = n-1$, then the polytope $P$ must be a simplex and hence simplicial. By Theorem~\ref{lem: embed polytope algo}, there is an algorithm with time polynomial in $\size(P)$ to create a graph $G= ([n], E, w)$ with rational positive edge weights and eigenspaces $\L_1 = \spanset\{\ones\}, \L_2 , \L_3$ for which $P$ is affinely equivalent to $P_{\{1, 3\}} \simeq P_{\{3\}}$. Consider the eigenspace ordering $
\L_1 < \L_2 < \L_3 $. By Theorem \ref{thm: main gale for D-A}, the facets of $P$ are in bijection with the minimal positively weighted $2$-graphical designs. A facet of $P$ is a simplex if and only if it contains exactly $d$ vertices. The polytope $P$ is then simplicial if and only if every minimal positively weighted $2$-graphical design has exactly $n-d = 1 + \dim(\L_2)$ vertices. Thus if we could determine whether a graph $G = ([n],E,w)$ has a $2$-graphical design of size smaller than the bound in Theorem~\ref{thm:general upper bound}, we could determine whether an arbitrary polytope $P$ was simplicial. Any $k$-design is also a $k'$-design for all $k'< k$, so the case of general $k$-designs must also be hard. 
\end{proof}

We note a further straightforward consequence of Theorem~\ref{lem: H-description simple} which we have not found explicitly stated in the literature.  We also credit Alexander E. Black for making the same observation independently.

\begin{corollary} \label{cor: largest facet is hard}
    The following problem is NP-hard.

\textbf{Input:} A rational $\cal V$-polytope $P = \conv(\{v_1,\ldots, v_n\}) $, $v_i \in \QQ^d$. 

\textbf{Output:} A facet $F = \conv(\{ v_{i_1}, \ldots, v_{i_j}\})$ of $P$ containing the maximum number of vertices of $P$ among all facets of $P$.
\end{corollary}

\begin{proof}
    A solution to this problem implies a solution to Corollary~\ref{cor: V-description simplicial}. Suppose there was an algorithm to find such a facet $F = \conv(\{ v_{i_1}, \ldots, v_{i_j}\})$.  If $j >d$, then $F$ is a non-simplex facet of $P$, hence $P$ is not simplicial.  If $j = d$, then every facet of $P$ has at most $d$ vertices by the maximality of $F$. Since a facet of a $d$-dimensional polytope must have at least $d$ vertices, it follows that every facet is the convex hull of exactly $d$ vertices, i.e. is a simplex.  Thus $P$ is simplicial. 
\end{proof}

Thus we can make the following translation to graphical designs. 
\begin{theorem} \label{thm: smallest design complexity}
The following problem is NP-hard. 

\textbf{Input:} \begin{enumerate}
    \item The Laplacian matrix $L= D-A \in \QQ^{n\times n}$ of a connected, positively weighted graph $G= ([n],E,w)$ with $m$ Laplacian eigenspaces $\L_1 = \spanset \{ \ones\} < \L_2 <\ldots < \L_m$, and
    \item an integer $k \in \{2,\ldots, m-1\}$. 
\end{enumerate} 

\textbf{Output:} A minimum cardinality positively weighted $k$-graphical design.
\end{theorem}

\begin{proof}
    The proof is nearly identical to that of Theorem~\ref{thm: is facet bound tight complexity}, so we will be brief.  A minimum cardinality positively weighted $k$-design is in bijection with a facet of the corresponding eigenpolytope containing a maximum number of eigenpolytope vertices.  Given a $\cal V$-polytope $P$, we can create a positively weighted graph $G$ which has an eigenpolytope that is affinely equivalent to $P$ in time polynomial in $\size(P)$. Thus an algorithm for this problem implies an algorithm for Corollary~\ref{cor: largest facet is hard}.
\end{proof}

Using a similar argument, we can show that it is \#P-complete to count the number of positively weighted $k$-designs of a graph. It was shown independently in \cite{DyerVertexEnum} and \cite{LinialEnum} that it is \#P-complete to count the number of facets of a polytope given by its vertices.

\begin{lemma}[\cite{DyerVertexEnum,LinialEnum}] \label{lem: V-description count facets} The following counting problem is \#P-complete.

\textbf{Input:} A rational $\cal V$-polytope $P = \conv(\{ v_1,\dots v_n\})$, $v_i \in \QQ^m$.

\textbf{Output:} Number of facets of $P.$
\end{lemma}

We again use Theorem~\ref{lem: embed polytope algo} to hide the facet-counting polytope problem in the problem of counting graphical designs.

\begin{theorem} \label{thm: count designs} The following counting problem is \#P-complete.

\textbf{Input:} \begin{enumerate}
    \item The Laplacian matrix $L= D-A \in \QQ^{n\times n}$ of a connected, positively weighted graph $G= ([n],E,w)$ which has $m$ Laplacian eigenspaces $\L_1 = \spanset \{ \ones\} < \L_2 <\ldots < \L_m$, and
    \item an integer $k \in \{2,\ldots, m-1\}$. 
\end{enumerate} 

\textbf{Output:} The number of minimal positively weighted $k$-graphical designs of $G$.
\end{theorem}
\begin{proof}
We will show that the counting problem in Lemma~\ref{lem: V-description count facets} is a sub-problem of this counting problem.  Let $P = \conv(\{v_1, \ldots, v_n\}\subset \QQ^d) $ be a polytope with $d \leq n-2$. We note that if $d = n-1$, then the polytope $P$ must be a simplex and hence has $n$ facets. By Theorem~\ref{lem: embed polytope algo}, there is a (strongly) polynomial time algorithm to create a graph $G= ([n], E, w \in \QQ_{>0}^E)$ with eigenspaces $\L_1 = \spanset\{\ones\}, \L_2 , \L_3$ for which $P$ is affinely equivalent to $P_{\{1, 3\}} \simeq P_{\{3\}}$. Consider the eigenspace ordering $\L_1 < \L_2 < \L_3 $. By Theorem~\ref{thm: main gale for D-A}, the facets of $P$ are in bijection with the minimal positively weighted $2$-graphical designs. Thus any enumeration of the minimal positively weighted $2$-graphical designs of $G$ would count the facets of $P$, hence this counting problem must be \#P-complete. This problem must be hard for any $k$ because a $k$-graphical design is also a $k'$-graphical design for all $k'< k$.
\end{proof}

Many algorithmic polytope questions are open when the dimension is not fixed; we refer  to \cite{KaibelPolyComplexSurvey} for a survey of algorithmic problems regarding polytopes. In particular, the complexity of providing an $\cal H$-description of a polytope given a $\cal V$-description is unknown. This is known as \emph{facet enumeration} or the \emph{convex hull problem}, and is widely believed to be difficult. A complexity result for the facet enumeration would provide a complexity result for the problem of listing all minimal $k$-graphical design, using the same framework as in the proofs of this section.

\begin{remark} \label{rem: polytime LP}
We note that it is polynomial time to find a not necessarily minimal $k$-graphical design. A solution to the linear program
\begin{equation}
\min c^\top x \qquad \text{   s.t.   } U_{\boldsymbol{k}} x = 0, \ones^\top x = 1, x \geq 0 
\label{eq: LP relaxation}
\end{equation}
provides a positively weighted $k$-graphical design for any cost vector $c\in \RR^n$, though the design need not be minimum or minimal.  
We recall that linear programming is polynomial time by \cite{KhachiyanLP, KarmarkarLP}, and moreover one can find an \emph{extreme point solution} in polynomial time \cite[Theorem 2.1.6]{IterCombOpti}. The Rank Lemma from linear programming guarantees some sparsity of extreme point solutions.
\begin{lemma}\cite[Lemma 1.2.3]{IterCombOpti}
    An extreme point solution $x^*$ of 
\begin{equation*}
\min c^\top x \qquad \text{   s.t.   } Ax = b, x \geq 0 
\end{equation*}
is zero on $\textup{corank}(A)$ coordinates.
\end{lemma}
Thus an extreme point solution of the LP in (\ref{eq: LP relaxation}) has $n - k$ zero coordinates, which is to say the design contains $k$ vertices. 
The one-norm is often used as a proxy for sparseness, but the constraints $\ones^\top x = 1, x \geq 0$ imply that $\|x\|_1 = 1$ for all feasible $x$ in (\ref{eq: LP relaxation}). While it thus theoretically makes no sense to minimize the one-norm, doing so in practice using MATLAB and Gurobi \cite{MATLAB,gurobi} has yielded very sparse graphical designs, often sparser than the guarantee of the Rank Lemma.  
\end{remark}

\printbibliography

\end{document}

%% file: CatherinesDefns.tex
\def\QQ{{\mathbb Q}}
\def\RR{{\mathbb R}}

\def\cal{\mathcal}

\def\l{\lambda}
\def\L{\Lambda}

\renewcommand{\phi}{\varphi}

\def\ones{\mathbbm{1}}

\def\conv{\operatorname {conv}}

\def\supp{\operatorname {supp}}

\def\Diag{\operatorname {Diag}}

\def\spanset{\operatorname {span}}
\def\rowspan{\operatorname {rowspan}}

\def\size{\operatorname {size}}
